\documentclass{article}
\usepackage{amsmath,amssymb,amscd,amsthm,amsfonts}
\usepackage{graphicx,subfigure}
\usepackage{psfig}
\usepackage{pstricks}

\newtheorem{theorem}{Theorem}
\newtheorem*{theoremp}{Theorem}
\newtheorem{lemma}[theorem]{Lemma}
\newtheorem{corollary}[theorem]{Corollary}
\newtheorem{conjecture}[theorem]{Conjecture}
\newtheorem*{conjecturep}{Conjecture}

\newcommand{\R}{\mathbb{R}}

\newcommand{\Z}{\mathbb{Z}}

\title{Equal coefficients and tolerance in coloured Tverberg partitions}
\author{Pablo Sober\'on\thanks{Partially supported by ERC Advanced Research Grant no 267165 (DISCONV)} \\ pablo.soberon@ciencias.unam.mx \and ------------------------------------------------ \\ Department of Mathematics \\ University College London \\ Gower Sreet, London WC1E 6BT \\ United Kingdom
\\ ---------------------------------------------}

\begin{document}
 
\maketitle

\begin{abstract}
 The coloured Tverberg theorem was conjectured by B\'ar\'any, Lov\'{a}sz and F\"uredi \cite{BFL89} and asks whether for any $d+1$ sets (considered as colour classes) of $k$ points each in $\R^d$ there is a partition of them into $k$ colourful sets whose convex hulls intersect.  This is known when $d=1,2$ \cite{BL92} or $k+1$ is prime \cite{BMZ11}.  In this paper we show that $(k-1)d+1$ colour classes are necessary and sufficient if the coefficients in the convex combination in the colourful sets are required to be the same in each class.  We also examine what happens if we want the convex hulls of the colourful sets to intersect even if we remove any $r$ of the colour classes.  Namely, if we have $(r+1)(k-1)d+1$ colour classes of $k$ point each, there is a partition of them into $k$ colourful sets such that they intersect using the same coefficients regardless of which $r$ colour classes are removed.  We also investigate the relation of the case $k=2$ and the Gale transform, obtaining a variation of the coloured Radon theorem.
\end{abstract}

\section{Introduction}

Tverberg's theorem is a very well known result in discrete geometry about partitions of points and intersection of convex hulls.  It says the following,

\begin{theoremp}[Tverberg's theorem \cite{Tve66}]
 Given a set of $(k-1)(d+1)+1$ points in $\R^d$, there is a partition of them into $k$ sets $A_1, A_2, \ldots, A_k$ such that their convex hulls intersect.
\end{theoremp}

This is a generalisation of Radon's theorem from 1921 \cite{Rad21}, which treats the case $k=2$.  Moreover, the number $(k-1)(d+1)+1$ cannot be improved.  A colourful generalisation of this theorem was conjectured by B\'ar\'any, F\"uredi and Lov\'{a}sz in 1989 \cite{BFL89}.  However, the conjecture below in its full strength was first proposed by B\'ar\'any and Larman \cite{BL92}.  The colourful version states the following,

\begin{conjecturep}[Coloured Tverberg's theorem]
Let $F_1, F_2, \ldots, F_{d+1}$ be sets (which we consider as colour classes) of $k$ points each in $\R^d$.  Then there is a partition of their union into $k$ pairwise disjoint colourful sets $A_1, A_2, \ldots, A_k$ such that their convex hulls intersect.
\end{conjecturep}

With a colourful set we mean a set that has exactly one element of each colour class.  Note that this is not a colourful version of Tverberg's theorem in the same way as the colourful versions of Helly's theorem or Carath\'eodory's theorem are \cite{Bar82}, since if all the colour classes are equal this does not yield the original theorem. By a colourful $k$-partition we refer to a family of $k$ pairwise disjoint colourful sets (even if every colour class has more than $k$ elements).  Historically this conjecture asked if there was a number $t=t(k,d)$ such that if $F_1, F_2, \ldots, F_{d+1}$ are sets of $t$ points each in $\R^d$, this type of $k$-partition always exists.  The existence of $t$ was first shown by B\'ar\'any, F\"uredi and Lov\'asz for $k=3$, $d=2$ \cite{BFL89}. The general case was settled by \v{Z}ivaljevi\'c and Vre\'cica, who showed that if $k$ was prime, $2k-1$ points were enough \cite{ZV92} (giving a bound of $4k-3$ points for all $k$).  Stated this way, the coloured Tverberg theorem is still open for most cases.

If $k+1$ is prime, this was solved by Blagojevi\'c, Matschke and Ziegler \cite{Z118} with topological methods (equivariant obstruction theory). At the core of their argument is the computation of degrees for a simplicial pseudomanifold; this was made explicit by Vre\'{c}ica and \v{Z}ivaljevi\'{c} \cite{Vrecica:2011bw}. By now there is also a purely geometrical proof of this fact by Matou\v{s}ek, Tancer and Wagner \cite{MTW11}, which still follows the same scheme. Blagojevi\'c et al.\ also gave an alternative proof using different, advanced topological tools (equivariant cohomology, index theory) in \cite{BMZ11}.
The optimal solution for the case $k+1$ prime gives the bound $t(k,d) \le 2k-2$ for all $k$, as noted in \cite{Z118}.  The case $d=1,2$ was solved by B\'ar\'any and Larman \cite{BL92} for any value of $k$.  For more references on this problem and historical notes we recomend \cite{MTW11}.

  The case $k=2$ (also known as the coloured Radon theorem) was originally solved by L\'aszl\'o Lov\'asz by constructing a function from $\mathbb{S}^d$ to $\R^d$ that depended on the pairs of points and then applying the Borsuk-Ulam theorem.  His proof is contained in \cite{BL92}.  Here we show a different proof of this fact using only basic linear algebra, which is a simplification of the methods that will be used later on.  In section \ref{gale} we present a third proof, using different methods (but still non-topological).

\begin{theoremp}[Coloured Radon]
 Given $d+1$ pairs of points $F_1, F_2, \ldots, F_{d+1}$ in $\R^d$, we can partition them into two disjoint colourful sets whose convex hulls have non-empty intersection.
\end{theoremp}

\begin{proof}
 Writing $F_i = \{x_i, y_i\}$ in an arbitrary way, the points $x_i - y_i$ are linearly dependent.  Then there is a linear combination $ \sum_{i=1}^{d+1} \alpha_i (x_i - y_i) = 0$ where not all the coefficients are $0$.  After proper relabelling of the points, we may suppose that all the coefficients are non-negative and, by a scalar multiplication, their sum is $1$.  Thus,
$$
\sum_{i=1}^{d+1} \alpha_i x_i = \sum_{i=1}^{d+1} \alpha_i y_i,
$$
as we wanted.
\end{proof}

This proof has the advantage that it gives an algorithm to find the partition and the point of intersection since it only involves finding a linear dependence.  Note that this proof not only gives the partition we wanted, but also shows that to find the point of intersection we may ask that the corresponding points have the same coefficients.  This also happens in Lov\'asz's poof, since the images of antipodal points in his construction have this property.  Thus it seems natural to ask whether this could also be possible in the coloured Tverberg Theorem.

To state this in a more precise way, let $F_1, F_2, \ldots, F_n$ be sets of at least $k$ points each in $\R^d$ and $A_1, A_2, \ldots, A_k$ a colourful $k$-partition of them.  We can denote the elements of each $A_i$ by $A_i = \{ x^i_j  :  x^i_j \in F_j \}$.  We say that the convex hulls of the $A_i$ intersect with equal coefficients if there are coefficients $\alpha_1, \alpha_2, \ldots, \alpha_n$ of a convex combination such that $\sum_{j} \alpha_j x^i_j $ is the same point for all $i$.

  In this paper we show that this extension of the colourful Tverberg theorem is not possible with $d+1$ colour classes, regardless of the number of points each class has.  The following theorem gives the optimal number of colour classes and the optimal number of points per class.  Namely,

\begin{theorem}\label{coeficientes-iguales}
 Let $F_1, F_2, \ldots, F_{n}$ be sets of $t$ points each in $\R^d$.  If $n=(k-1)d+1$ and $t = k$, we can find a colourful $k$-partition $A_1, A_2, \ldots, A_k$ of them such that their convex hulls intersect with equal coefficients.  Moreover,  if $n < (k-1)d+1$ this theorem is false regardless of the value of $t$.
\end{theorem}

In section \ref{remarks} we show that for $n=(k-1)d+1$ and $t=k$, there are at least $(k-1)!^{(k-1)d}$ such partitions.  It is conjectured that for the classical Tverberg theorem there are always at least $(k-1)!^d$ good partitions.  This is known as Sierksma's conjecture or the Dutch cheese conjecture.  Bounds for the number of Tverberg partitions have been obtained when $k$ is a prime power (\cite{Vucic:1993be}, \cite{Hell:2007tp}). The only non-trivial case that has been solved is $d=2, k=3$ by Hell \cite{Hell:2008em}.  We find it surprising that in this aspect the classical Tverberg theorem seems more resistant.

Recently many theorems with tolerance have appeared in this kind of settings.  We say that a property $P$ of points in $\R^d$ is true in $F$ with tolerance $r$ if $P(F\backslash C)$ is true for all $C \subset F$ of size $r$.  For example, $P$ can be ``captures the origin''.   There are now versions of the classical Helly, Carath\'eodory and Tverberg theorems with a tolerance condition \cite{MO11}, \cite{SS-unp}.  We show that the previous theorem also has a version with tolerance, but in this case we do not know that the number of colour classes is optimal.

\begin{theorem}\label{teorema-toleardo}
Let $d \ge 2$, $n= (r+1)(k-1)d+1$ and $F_1, F_2, \ldots, F_n$ be sets of $k$ points each in $\R^d$.  Then we can find a colourful $k$-partition $A_1, A_2, \ldots, A_k$ of them such that for any set $C$ of $r$ colour classes, the convex hulls of $A_1 \backslash C$, $A_2 \backslash C$, $\ldots$, $A_k \backslash C$ intersect with equal coefficients.
\end{theorem}

The proofs of Theorems \ref{coeficientes-iguales} and \ref{teorema-toleardo} are given in section \ref{pruebas}. If $k=2$, theorem \ref{teorema-toleardo} gives a nice version of the coloured Radon theorem with tolerance.

\begin{corollary}[Coloured Radon theorem with tolerance]
 Let $d \ge 2$, and $F_1$, $F_2$, $\ldots$, $F_{(r+1)d+1}$ be pairs of points in $\R^d$.  Then we can split them into two disjoint colourful sets $A_1, A_2$  such that, if we remove any $r$ colour classes, the convex hulls of what is left in $A_1$ and $A_2$ intersect.
\end{corollary}

It would be interesting to know if the number of parts in the corollary above is enough for a coloured Tverberg theorem with tolerance (without equal coefficients).  Namely,

\begin{conjecture}[Coloured Tverberg theorem with tolerance]
 Let $d \ge 2$, and $F_1, F_2, \ldots, F_{(r+1)d+1}$ be sets of $k$ points each in $\R^d$.  Then we can split them into $k$ disjoint colourful sets $A_1, A_2, \ldots, A_k$  such that, if we remove any $r$ colour classes, the convex hulls of what is left in $A_1, A_2, \ldots, A_k$ intersect.
\end{conjecture}

In section \ref{gale} we investigate the relation between the Gale transform and the coloured Radon theorem.  With this we are able to obtain the following variation.

\begin{theorem}\label{variacion-rara}
Given a set of $k+d+2$ points in $\R^d$ of $k$ possible colours, we can find two disjoint sets $A$ and $B$ such that they have the same number of points of each colour and their convex hulls intersect.
\end{theorem}

Note that the theorem above does not imply coloured Radon.  This is due to the fact that the theorem does not take into consideration how the colours are distributed.  However, the coloured Radon theorem can be proved with the same method if this extra information is considered.

\section{Preliminaries}

The proof of theorem \ref{coeficientes-iguales} will be in the same spirit of Sarkaria's proof of Tverberg's theorem \cite{Zar92} but without lifting the points to $\R^{d+1}$.  We are able to use the additional structure on the type of partitions we want in order to reduce the number of dimensions we need.

Let $F_1, F_2, \ldots, F_n$ be $n$ sets of $t$ points each in $\R^d$ and $k \le t$ a positive integer.  For convenience we consider each $F_j$ as an ordered set and denote its elements by $F_j = \{ z(j)_1, z(j)_2, \ldots, z(j)_t\}$.  This order given to each $F_j$ is not important at the moment, but it will be when counting the number of good partitions.  Denote by $\Sigma_{k,t}$  the set of injective functions from $[k]=\{1,2,\ldots, k\}$  to $[t]$.  We can consider the colourful $k$-partitions $(A_1, A_2, \ldots, A_k)$ as vectors $(\sigma_1, \sigma_2, \ldots, \sigma_n)$ in $(\Sigma_{k,t})^n$ by defining
$$
\sigma_j (i) = m \ \mbox{if and only if } x^i_j = z(j)_m
$$
or equivalently
$$
A_i = \{ x^i_j : x^i_j = z(j)_{\sigma_j (i)}\}
$$
We call $(\sigma_1, \sigma_2, \ldots, \sigma_n)$ the function representation of $(A_1, A_2, \ldots, A_k)$.  If $t=k$ we call this the permutation representation of $(A_1, A_2, \ldots, A_k)$.

To see this in a simpler way, write the elements of each $F_j$ $k$ times in $k$ rows.  Then choose from the first row the elements of $A_1$, in the second row the elements of $A_2$ and so on.  Then the function representation of $(A_1, A_2, \ldots, A_k)$ becomes apparent, as in the figure below.  The reverse operation can also be performed.

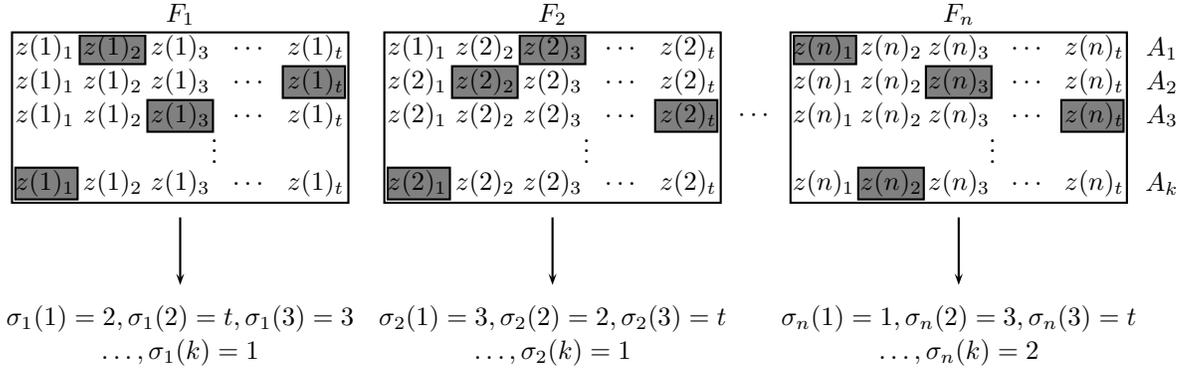
\psset{unit=0.45cm}
\begin{pspicture}(0,15)
\def\figurafuncion{
\psframe[fillstyle=solid, fillcolor=gray](2,4.5)(4,5.4)
	\psframe[fillstyle=solid, fillcolor=gray](8,3.5)(9.9,4.5)
	\psframe[fillstyle=solid, fillcolor=gray](4,2.5)(6,3.5)
	\psframe[fillstyle=solid, fillcolor=gray](0.1,0.5)(2,1.5)
	\psframe[fillstyle=solid, fillcolor=gray](15,4.5)(17,5.4)
	\psframe[fillstyle=solid, fillcolor=gray](13,3.5)(15,4.5)
	\psframe[fillstyle=solid, fillcolor=gray](19,2.5)(20.9,3.5)
	\psframe[fillstyle=solid, fillcolor=gray](11.1,0.5)(13,1.5)
	\psframe[fillstyle=solid, fillcolor=gray](23.1,4.5)(25,5.4)
	\psframe[fillstyle=solid, fillcolor=gray](27,3.5)(29,4.5)
	\psframe[fillstyle=solid, fillcolor=gray](31,2.5)(32.9,3.5)
	\psframe[fillstyle=solid, fillcolor=gray](25,0.5)(27,1.5)
	\rput(1,5){$z(1)_1$}
	\rput(1,4){$z(1)_1$}
	\rput(1,3){$z(1)_1$}
	\rput(1,1){$z(1)_1$}
	\rput(3,5){$z(1)_2$}
	\rput(3,4){$z(1)_2$}
	\rput(3,3){$z(1)_2$}
	\rput(3,1){$z(1)_2$}
	\rput(5,5){$z(1)_3$}
	\rput(5,4){$z(1)_3$}
	\rput(5,3){$z(1)_3$}
	\rput(6,2.2){$\vdots$}
	\rput(5,1){$z(1)_3$}
	\rput(7,5){$\ldots$}
	\rput(7,4){$\ldots$}
	\rput(7,3){$\ldots$}
	\rput(7,1){$\ldots$}
	\rput(9,5){$z(1)_t$}
	\rput(9,4){$z(1)_t$}
	\rput(9,3){$z(1)_t$}
	\rput(9,1){$z(1)_t$}
	\psframe(0,0.4)(10,5.5)
	\psline[arrows=->](5,0)(5,-2)
	\rput(5,-3){$\sigma_1 (1) = 2, \sigma_1 (2) = t, \sigma_1(3)=3$}
	\rput(5,-4){$ \ldots,  \sigma_1(k) = 1$}
	\rput(5,6){$F_1$}
	\rput(12,5){$z(1)_1$}
	\rput(12,4){$z(2)_1$}
	\rput(12,3){$z(2)_1$}
	\rput(12,1){$z(2)_1$}
	\rput(14,5){$z(2)_2$}
	\rput(14,4){$z(2)_2$}
	\rput(14,3){$z(2)_2$}
	\rput(14,1){$z(2)_2$}
	\rput(16,5){$z(2)_3$}
	\rput(16,4){$z(2)_3$}
	\rput(16,3){$z(2)_3$}
	\rput(16,1){$z(2)_3$}
	\rput(18,5){$\ldots$}
	\rput(18,4){$\ldots$}
	\rput(18,3){$\ldots$}
	\rput(18,1){$\ldots$}
	\rput(17,2.2){$\vdots$}
	\rput(20,5){$z(2)_t$}
	\rput(20,4){$z(2)_t$}
	\rput(20,3){$z(2)_t$}
	\rput(20,1){$z(2)_t$}
	\psframe(11,0.4)(21,5.5)
	\psline[arrows=->](16,0)(16,-2)
	\rput(16,-3){$\sigma_2 (1) = 3, \sigma_2 (2) = 2, \sigma_2(3)=t$}
	\rput(16,-4){$ \ldots,  \sigma_2(k) = 1$}
	\rput(22,3){$\cdots$}
	\rput(16,6){$F_2$}
	\rput(24,5){$z(n)_1$}
	\rput(24,4){$z(n)_1$}
	\rput(24,3){$z(n)_1$}
	\rput(24,1){$z(n)_1$}
	\rput(26,5){$z(n)_2$}
	\rput(26,4){$z(n)_2$}
	\rput(26,3){$z(n)_2$}
	\rput(26,1){$z(n)_2$}
	\rput(28,5){$z(n)_3$}
	\rput(28,4){$z(n)_3$}
	\rput(28,3){$z(n)_3$}
	\rput(28,1){$z(n)_3$}
	\rput(30,5){$\ldots$}
	\rput(30,4){$\ldots$}
	\rput(30,3){$\ldots$}
	\rput(30,1){$\ldots$}
	\rput(29,2.2){$\vdots$}
	\rput(32,5){$z(n)_t$}
	\rput(32,4){$z(n)_t$}
	\rput(32,3){$z(n)_t$}
	\rput(32,1){$z(n)_t$}
	\psframe(23,0.4)(33,5.5)
	\psline[arrows=->](28,0)(28,-2)
	\rput(28,-3){$\sigma_n (1) = 1, \sigma_n (2) = 3, \sigma_n(3)=t$}
	\rput(28,-4){$ \ldots,  \sigma_n(k) = 2$}
	\rput(28,6){$F_n$}
	\rput(34,5){$A_1$}
	\rput(34,4){$A_2$}
	\rput(34,3){$A_3$}
	\rput(34,1){$A_k$}
	}
	\rput(-3,6){\figurafuncion}
\end{pspicture}

Let $u_1, u_2, \ldots, u_k$ be the vertices of a regular simplex in $\R^{k-1}$ centred at the origin.  We use these sets to represent the distribution in the $k$-partition.  Given $\sigma \in \Sigma_{k,t}$ we define $F_j (\sigma) \in \R^{(k-1)d}$ as 
$$
F_j (\sigma) = \sum_{i=1}^{k} u_i \otimes z(j)_{\sigma (i)}
$$
where $\otimes$ represents the tensor product.
\begin{lemma}\label{lema-tecnico}
Let $F_1, F_2, \ldots, F_n$ be sets of $t$ points each in $\R^d$ and $(A_1, A_2, \ldots, A_k)$ be a colourful $k$-partition of them.  Then for coefficients $\alpha_1, \alpha_2, \ldots, \alpha_n$ we have that $\sum_{j=1}^n \alpha_j x^i_j$ is the same point for all $i$ if and only if $\sum_{j=1}^n \alpha_j F_j(\sigma_j) = 0$, where $(\sigma_1, \sigma_2, \ldots, \sigma_n)$ is the function representation of $(A_1, A_2, \ldots, A_k)$.
\end{lemma}

\begin{proof}
 It suffices to prove this for $d=1$ since we can repeat the same argument for each coordinate in $\R^d$.  In this case $u_i \otimes z(j)_m$ is simply $z(j)_m u_i$.  Note that, for scalars $\beta_1, \beta_2, \ldots, \beta_k$, we have that $\sum_{i=1}^k \beta_i u_i = 0$ if and only if $\beta_1 = \beta_2 = \ldots = \beta_k$.  Thus $\sum_{j=1}^n \alpha_j F_j(\sigma_j) = 0$ if and only if $\sum_{j=1}^n \alpha_j z(j)_{\sigma_j (i)}$ is the same for all $i$, as we wanted.
\end{proof}

\section{Proofs of theorems \ref{coeficientes-iguales} and \ref{teorema-toleardo}}\label{pruebas}

\begin{proof}[Proof of theorem \ref{coeficientes-iguales}]
We first show that for $n = (k-1)d+1$ and $t=k$ there are such coefficients.  For this it suffices to note that each of the sets $F_j (\Sigma_{k,k})$ captures the origin.  We have $n$ sets that capture the origin in $\R^{n-1}$, so by the B\'ar\'any colourful version of Carath\'eodory's theorem \cite{Bar82}, there are permutations $\sigma_1, \sigma_2, \ldots, \sigma_n$ such that the set $\{ F_1 (\sigma_1), F_2 (\sigma_2), \ldots, F_n (\sigma_n)\}$ captures the origin.  Using lemma \ref{lema-tecnico} we are done.

Now suppose that $n \le (k-1)d$.  If $\sigma_1$ is fixed, note that by varying $F_1$, then $F_1(\sigma_1)$ can be any point of $\R^{(k-1)d}$.  Suppose we are given $F_2, F_3, \ldots, F_n$ and we want to find all the choices for $F_1$ that satisfy the theorem.  Given any coloured $k$-partition $(A_1, A_2, \ldots, A_k)$, by applying the same permutation to each colour class, we can assume that in its function representation $\sigma_1$ is the identity.  If we can find injective functions $\sigma_2, \sigma_3, \ldots, \sigma_n$ such that the set $\{ F_1 (\sigma_1), F_2 (\sigma_2), \ldots, F_n (\sigma_n)\}$ captures the origin then $F_1 (\sigma_1)$ must be in the $(n-1)$-flat generated by ${0, F_2(\sigma_2),\ldots, F_n(\sigma_n)}$.  Note that there is only a finite number of possible choices for $\sigma_2, \sigma_3, \ldots, \sigma_n$.  Since a finite number of flats of codimension at least $1$ cannot cover $\R^{(k-1)d}$, we are done. 
\end{proof}

To show the versions with tolerance with this method we would need some version of the colourful Carath\'eodory theorem with tolerance.  Conveniently this was done in Lemma 1 in the proof of Tverberg's theorem with tolerance \cite{SS-unp}.  Here we re-write this lemma to fit the current notation.

\begin{lemma}\label{lema-tolerado}
Let $p \ge 1$ and $r \ge 0$ be integers, $n=(r+1)p+1$, $F_1, F_2, \ldots, F_n \subset \R^p$ subsets of $\R^p$ that capture the origin, $F = \bigcup_{j=1}^n F_j$ and $G$ a group with $|G| \le p$.  Suppose there is an action of $G$ in each of $F_1, F_2, \ldots, F_n$ such that the following holds.
\begin{itemize}
	\item For all $x \in F_j$, $Gx$ captures the origin, for all $j$.
	\item For all $A \subset F$, if $A$ captures the origin then so does $gA$, for all $g \in G$.
\end{itemize}
Then we can find elements $x_1 \in F_1, x_2 \in F_2, \ldots, x_n \in F_n$ such that $\{ x_1, x_2, \ldots, x_n\}$ captures the origin with tolerance $r$.
\end{lemma}

\begin{proof}[Proof of theorem \ref{teorema-toleardo}]
Note that there is an action of $\Sigma_{k,k}$ (as the symmetric group) in $F_j(\Sigma_{k,k})$ given by $\sigma F_j (\tau) = F_j (\sigma \tau)$.  In particular, this induces an action of $\Z_k$ given by $m F_j(\tau) = F_j (\beta^m \tau)$ where $\beta$ is a cycle of length $k$.  Consider $p=(k-1)d$ and $F_1 (\Sigma_{k,k})$,  $F_2 (\Sigma_{k,k})$, $\ldots$, $F_n(\Sigma_{k,k})$ with their action of $\Z_k$ (all using the same $\beta$).  If $d \ge 2$, these sets with their actions of $\Z_k$ satisfy the conditions of lemma \ref{lema-tolerado}.  Using this lemma and lemma \ref{lema-tecnico} we obtain the result.
 
\end{proof}

\section{Remarks on the proof}\label{remarks}

In the general case of the coloured Tverberg theorem, it seems tempting to extend Lov\'asz's argument for the case $k >2$.  Namely,  we consider the $(d+1)$-fold join $U = \underbrace{G \ast G \ast \cdots \ast G}_{d+1}$ of a discrete set with $k$ elements and try to prove that for any linear function $f: U \longrightarrow \R^d$ there are points $x_1, x_2, \ldots, x_k$ in disjoint faces of $U$ such that $f(x_1)=f(x_2)=\ldots = f(x_k)$.  Then it seems natural to use the combinatorial properties of $U$ (or a further refinement of this space) to treat this as an equivariant topology problem with the natural action of $G$ in $U$, as it is done in cases like the topological Tverberg theorem \cite{Lon01}.  However, a simple approach of this kind is bound to give points $x_1, x_2, \ldots, x_k$ which have the same coefficients when written as combination of the $d+1$ copies of $G$.  Since we showed that for this case at least $(k-1)d+1$ colour classes are necessary, other topological methods are needed for the general case.

In the proof of theorem \ref{coeficientes-iguales}, the final argument was that each $F_j (\Sigma_{k,k})$ captured the origin.  However, a smaller set would suffice, namely $F_j (\{ \beta, \beta^2, \ldots, \beta^k\})$ captures the origin for any $\beta$ a cycle of length $k$.  For simplicity we can take $\beta$ the permutation that shifts every element to the right, except the last one.  Fix the order of $F_1$ and assign a cyclic order (that is, an order up to iterated applications of $\beta$) to each $F_j$ with $j>1$.  Note that the partitions we obtain for each way to assign cyclic orders to these families of points are all different.  Thus, there are at least $(k-1)!^{(k-1)d}$ of these partitions.  This is similar to the conjectured $(k-1)!^d$ different partitions for the typical Tverberg theorem in the sense that both are roughly $(k-1)!^{{m}/{k}}$ where $m$ is the number of points that are given in the theorem.  

The last step towards showing that $(k-1)d+1$ colour classes are necessary for Theorem \ref{coeficientes-iguales} relies on the fact that a finite number of flats of codimension at least $1$ cannot cover the whole space.  However, more can be said.  Let $n\le (k-1)d$ and $F_1, F_2, \ldots, F_n$ be sets of $t$ random points each in $\R^d$ where each point is chosen according to a (possibly different) distribution where hyperplanes have measure $0$.  Then, with probability $1$, there is no colourful $k$-partition $A_1, A_2, \ldots, A_k$ such that the convex hulls of the $A_i$ intersect with equal coefficients.

\section{Coloured Radon and the Gale transform}\label{gale}

Here we present a third proof of the coloured Radon, again without topological tools.  This is done via the Gale transform.  The Gale transform of a set of $n$ points $a_1, a_2, \ldots, a_n$ in $\R^d$ that are not all contained in a hyperplane is a set of $n$ points $b_1, b_2, \ldots, b_n$ in $\R^{n-d-1}$ such that the following two conditions hold
\begin{itemize}
	\item $\sum_i b_i = 0$ and
	\item for every two disjoint subsets $X, Y \subset [n]$, the convex hull of the sets $\{ a_i :  i \in X\}$ and $\{ a_i  :  i \in Y \}$ intersect if and only if there is a hyperplane $H$ through the origin in $\R^{n-d-1}$ that leaves $\{ b_i  :  i \in X\}$ in one (closed) side, $\{b_i  :  i \in Y\}$ in the other (closed) side and goes through every other $b_i$.
\end{itemize}

See, for example, \cite{grunbaum2003convex} for a complete exposition.  We are now ready to prove the coloured Radon theorem.

\begin{proof}
Let $F_1, F_2, \ldots, F_{d+1}$ be the sets of pairs of points in $\R^d$ and $F = \cup_i F_i$.  Let $J$ be the Gale transform of $F$.  Then $J$ consists of $d+1$ pairs of points $J_1, J_2, \ldots, J_{d+1}$ in $\R^{d+1}$.  We want to find a hyperplane $H$ through the origin that splits every pair in $J$ simultaneously.  Let $m_i$ be the midpoints of the $J_i$.  Note that since the sum of all the points in $J$ is $0$, then the sum of all $m_i$ is $0$.  With this the hyperplane $H$ through $m_1, m_2, \ldots, m_{d+1}$ goes through the origin, and splits each pair.
\end{proof}

With this idea in mind we can prove theorem \ref{variacion-rara}.  We use the ham-sandwich theorem (see \cite{Mat2003libro}) in addition to the Gale transform.

\begin{proof}[Proof of theorem \ref{variacion-rara}]
Consider a set $X$ of $k+d+2$ points in $\R^d$ each painted with one of $k$ possible colors.  Its Gale transform is a set $Y$ in $\R^{k+1}$.  If we consider the origin of $\R^{k+1}$ painted with a new color, we have $k+1$ coloured finite sets.  Using the ham-sandwich theorem, there is a hyperplane that splits them all evenly.  This hyperplane must go through the origin, and thus gives the two subset of $X$ we we looking for.
\end{proof}

If $k=1$ this gives the following corollary

\begin{corollary}
Given $d+3$ points in $\R^d$, there are two disjoint subset $A$ and $B$ of the same size such that their convex hulls intersect.
\end{corollary}

It would be interesting to find an analogous statement for Tverberg partitions.  Namely, finding the smallest $n=n(d,k)$ such that for every set of $n$ points in $\R^d$ we can find $k$ pairwise disjoint subsets of the same size whose convex hulls intersect.   Clearly $(k-1)(d+1) < n \le k(d+1)$.

\section{Acknowledgements}

We would like to thank Imre B\'ar\'any for the helpful discussions and observations on this subject and his help on simplifying the proof of the coloured Radon theorem in the introduction.
\bibliographystyle{amsplain}

\bibliography{referencias.bib}
\end{document}